\documentclass{article}

\usepackage{array, amsmath,amssymb,amsthm}
\usepackage{amscd}
\usepackage[english]{babel}
\usepackage{hyperref}
\usepackage{color}
\usepackage{url}
\usepackage{marginnote}
 \usepackage[all]{xy}
 
\newtheorem{theorem1}{Special Theorem}

\newtheorem{theorem}[theorem1]{Theorem}
\newtheorem{notation}[theorem1]{Notation}

\newtheorem{proposition}[theorem1]{Proposition}
\newtheorem{consequence}[theorem1]{Consequence}
\newtheorem{definition}[theorem1]{Definition}
\newtheorem{remark}[theorem1]{Remark}
\newtheorem{example}[theorem1]{Example}
\newtheorem{lemma}[theorem1]{Lemma}
\newtheorem{corollary}[theorem1]{Corollary}
\newcommand{\segment}{\!\upharpoonright \!}

\newcommand{\N}{{\mathbb N}}    % Non negative integers
\newcommand{\Z}{{\mathbb Z}}     % Integers
\newcommand{\FINZ}{{\+P_{<\omega}(\Z)}}

\newcommand{\Zhat}{\widehat{\Z}}    

\newcommand{\val}{\textit{Val}}     
\newcommand{\suc}{\textit{Suc}}     

\newcommand{\lcm}{\textit{lcm}}     
     
\newcommand{\divi}{\mid}

\def \go {\leavevmode
        \raise.3ex\hbox{$\scriptscriptstyle\langle\!\langle\,$}%
        ~\ignorespaces}

\def \gf {\relax \ifhmode \unskip~\else \leavevmode \fi
        \raise.3ex\hbox{$\scriptscriptstyle\,\rangle\!\rangle$}}
{\end{list}}%

{\end{list}}%

\begin{document}
%\color{cyan}
%\mainmatter
\begin{center}
 {\Large \bf Arithmetical Congruence Preservation:
\\from Finite to Infinite}
%\titlerunning{Arithmetical Congruence preserving Functions}

\vskip 10pt
{\bf Patrick C\'EGIELSKI%
\newcounter{thanks}
\setcounter{thanks}{\value{footnote}}
\footnote{Partially supported by TARMAC ANR agreement  12 BS02 007 01.}}\\
{\small  LACL, EA 4219, Universit\'e Paris-Est Cr\'eteil, France}\\
{\tt cegielski@u-pec.fr}\\ 
%----------------
\vskip 10pt
{\bf Serge GRIGORIEFF
\setcounter{thanks}{\value{footnote}}%
\footnotemark[\value{thanks}]}\\
{\small LIAFA, CNRS and Universit\'e Paris-Diderot, France}\\
{\tt seg@liafa.univ-paris-diderot.fr}\\ 
%----------------
\vskip 10pt
{\bf Ir\`ene  GUESSARIAN
\setcounter{thanks}{\value{footnote}}%
\footnotemark[\value{thanks}]
\footnote{Emeritus at UPMC Universit\'e Paris 6. Corresponding author}}\\
{\small LIAFA, CNRS and Universit\'e Paris-Diderot, France}\\
{\tt ig@liafa.univ-paris-diderot.fr}\\ 
\end{center}
\vskip 30pt

%\maketitle
%\begin{center}
%{\sl To Yuri, on his 75th birthday, with thanks  for many stimulating discussions on Logic and Computation}
%\end{center}

\begin{abstract}
Various problems on integers lead to the  class of   congruence preserving functions on rings,  i.e.  functions verifying $a-b$ divides $f(a)-f(b)$ for all $a,b$.
We characterized these classes of functions
in terms of sums of rational polynomials (taking only integral values)
and the function giving the least common multiple of $1,2,\ldots,k$. The tool used to obtain these characterizations is ``lifting": if $\pi\colon X\to Y$ is a surjective morphism, and $f$ a function on $Y$
a lifting  of $f$ is a function $F$ on $X$ such that $\pi\circ F=f\circ\pi$.
In this paper we relate the finite and infinite notions by proving that the finite case
can be lifted to the infinite one.
For $p$-adic and profinite integers we  get similar characterizations via lifting.
We also prove that lattices of recognizable subsets of $\Z$ are stable under inverse image by congruence preserving functions.

\end{abstract}
\bibliographystyle{plain}

%%%%%%%%%%%%%%%%%%%%%%%%%%%%
%%%%%%%%%%%%%%%%%%%%%%%%%%%%
%%%%%%%%%%%%%%%%%%%%%%%%%%%%
\section{Introduction}\label{s:intro}
%%%%%%%%%%%%%%%%%%%%%%%%%%%%
%%%%%%%%%%%%%%%%%%%%%%%%%%%%
%%%%%%%%%%%%%%%%%%%%%%%%%%%%
%
%
A function $f$ (on $\N$ or $\Z$) is said to be congruence preserving if $a-b$ divides $f(a)-f(b)$.
Polynomial functions are obvious examples of congruence preserving functions.
In \cite{cgg14,cgg14a} we characterized this notion
(which we named ``functions having the integral difference ratio property'')
for functions $\N\to\Z$ and $\Z\to\Z$.
In \cite{cgg15} we extended the characterization to functions 
$\Z/n\Z\to\Z/m\Z$
(for a suitable extension of the notion of congruence preservation).

In the present paper, we prove in \S\ref{s:lift} that 
every congruence preserving function $\Z/n\Z\to\Z/n\Z$
can be lifted to a congruence preserving function $\N\to\N$
(i.e. it is the projection of such a function).
As a corollary  %in \S\ref{ss:liftZnZ}
(i) we show that such a lift also works replacing $\N$  with  $\Z/qn\Z$
and (ii) and we give an alternative proof of a representation
(obtained in \cite{cgg15})
of congruence preserving functions $\Z/n\Z\to\Z/m\Z$
as linear sums of ``rational'' polynomials.

In \S\ref{s:profinite} we consider the rings of $p$-adic integers
(resp. profinite integers)
and prove that  congruence preserving functions are inverse limits of 
congruence preserving functions on the $\Z/p^k\Z$ (resp. on the $\Z/n\Z$).
Considering the Mahler representation of continuous functions by
Newton series, we prove that congruence preserving functions
correspond to those series for which the linear coefficient with rank $k$
is divisible by the least common multiple of $1,\ldots,k$.

We proved in \cite{cgg14ipl} that lattices of regular subsets of $\N$ are closed under inverse image by congruence preserving functions: in  \S\ref{s:lattices}, we  extend this result to functions $\Z\to\Z$.

%%%%%%%%%%%%%%%%%%%%%%%%%%%%
%%%%%%%%%%%%%%%%%%%%%%%%%%%%
%%%%%%%%%%%%%%%%%%%%%%%%%%%%
\section{Congruence preservation: exchanging finite and infinite}
\label{s:lift}
%%%%%%%%%%%%%%%%%%%%%%%%%%%%
%%%%%%%%%%%%%%%%%%%%%%%%%%%%
%%%%%%%%%%%%%%%%%%%%%%%%%%%%
%
%
We characterize congruence preserving functions on  $\Z/n\Z$ by
 first lifting each such function  into a
congruence preserving function $\N \to \N$. In a second step, we use our characterization of congruence preserving functions $\N\to\Z$ to characterize 
the congruence preserving functions $\Z/n\Z \to \Z/n\Z$.

\begin{definition}\label{def:idr}
Let $X$ be a subset of a commutative (semi-)ring $(R,+,\times)$.
A function $f\colon X\to R$ is said to be congruence preserving if
\begin{eqnarray*}\label{eq:idr}
\forall x,y\in X\quad \exists d\in R\quad f(x)-f(y)\ =d(x-y)\,,\quad \text{\rm i.e.\ } x-y \text{ divides } f(x)-f(y)\,.
\end{eqnarray*}
\end{definition}
\begin{definition}[Lifting]\label{def:lift} 
Let $\sigma\colon X\to N$ and $\rho\colon Y\to M$ be surjective maps.
A function $F\colon X \to Y$  is said to be a 
{\em $(\sigma,\rho)$-lifting } of a function $f\colon N \to M$
(or simply {\em lifting} if $\sigma,\rho$ are clear from the context)
if the following diagram commutes:
\\\centerline{\quad$\begin{CD}
 X      @>\text{\normalsize$F$}>>   Y  \\
@V\text{\normalsize$\sigma$}VV    
@VV\text{\normalsize$\rho$} V\\
N    @>\text{\normalsize$f$}>>  M
\end{CD}$
\qquad
i.e. \quad $\rho\circ F= f\circ \sigma$\,.}
\end{definition}

We will  consider  elements of $\Z/k\Z$ as  integers
and vice versa via the following maps.
\begin{notation}\label{def:pi n}
1.  Let $\pi_k\colon \Z\to\Z/k\Z$ be the canonical surjective homomorphism  associating to an integer its class in $\Z/k\Z$.
\\
2. Let $\iota_k\colon\Z/k\Z\to\N$ be the injective map 
 associating to an element $x\in\Z/kZ$
its representative in $\{0,\ldots,k-1\}$. 
\\
3. Let $\pi_{n,m} \colon  \Z/n\Z \to \Z/m\Z$ be the 
 map $\pi_{n,m}=\pi_m\circ\iota_n$. In case $m$ divides $n$, $\pi_{n,m}$ is a surjective homomorphism.
 
 If $m\leq n$ let $\iota_{m,n} \colon  \Z/m\Z \to \Z/n\Z$ be the 
injective map $\iota_{m,n}=\pi_n\circ\iota_m$.
\end{notation}
\begin{lemma}\label{l:utile} If $m$ divides $n$, $\pi_{m}=\pi_{n,m}\circ \pi_{n}$.
\end{lemma}
The next theorem insures that congruence preserving functions
$\Z/n\Z\to\Z/n\Z$ can be lifted to congruence preserving functions
$\N\to\Z$.
\begin{theorem}
[Lifting functions $\Z/n\Z\to\Z/n\Z$ to $\N\to\N$]
\label{thm:lift nn}  
Let $f\colon \Z/n\Z \to \Z/n\Z$ with $m\geq 2$. The following conditions are equivalent:
\begin{enumerate}
\item
 $f$ is congruence preserving.
\item
$f$ can be $(\pi_n,\pi_n)$-lifted to a congruence preserving 
function $F:\N\to\N$.
\end{enumerate}
\end{theorem}
In view of applications in the context of $p$-adic and profinite integers,
we state  and prove a slightly more general version  with an extended notion of congruence preservation defined below.

\begin{definition}\label{def:cp ZnZm}
A function $f:\Z/n\Z\to\Z/m\Z$ is congruence preserving if
\begin{equation}\label{eq:cp ZnZm}
\text{for all $x,y\in \Z/n\Z$,\quad
 $\pi_{n,m} (x-y)$ divides $f(x)-f(y)$ in $\Z/m\Z$}\,.
\end{equation}
\end{definition}
\begin{theorem} [Lifting functions $\Z/n\Z\to\Z/n\Z$ to $\N\to\N$]\label{thm:lift nm}  
Let $f\colon \Z/n\Z \to \Z/m\Z$ with $m$ divides $n$ and $m\geq 2$. The following conditions are equivalent:
\begin{enumerate}
\item
 $f$ is congruence preserving.
\item
$f$ can be $(\pi_n,\pi_m)$-lifted to a congruence preserving function $F:\N\to\N$.
\item
$f$ can be $(\pi_n,\pi_m)$-lifted to a congruence preserving function $F:\N\to\Z$.
\end{enumerate}
\end{theorem}
\begin{proof}
$(2)\Rightarrow(3)$ is trivial.
\\
$(3)\Rightarrow(1)$.
Assume $f$ lifts to the congruence preserving function $F:\N\to\Z$. The following diagram commutes\\
$$\begin{CD}
 \N      @>\text{\normalsize$F$}>>   \Z  \\
@V\text{\normalsize$\pi_n$}VV        @VV\text{\normalsize$\pi_m$}V\\
\Z/n\Z    @>\text{\normalsize$f$}>>  \Z/m\Z
\end{CD} \quad \text{and thus \ }
\left\{\begin{array}{rcl}
\pi_m\circ F&=& f\circ \pi_n\\ 
 f&=&\pi_m\circ F\circ\iota_n
\end{array}\right.
$$
Let $x,y\in\Z/n\Z$.
As $F$ is congruence preserving, 
$\iota_n(x)-\iota_n(y)$ divides $F(\iota_n(x))-F(\iota_n(y))$,
hence  $F(\iota_n(x))-F(\iota_n(y))=(\iota_n(x)-\iota_n(y))\,\delta$.
Since $\pi_m$ is a morphism 
and $\pi_m\circ\iota_n=\pi_{n,m}$, we get
$\pi_m(F(\iota_n(x)))-\pi_m(F(\iota_n(x)))=\pi_{n,m}(x-y)\,\pi_{n,m}(\delta)$.
As $F$ lifts $f$ we have 
$\pi_m(F(\iota_n(x)))-\pi_m(F(\iota_n(y)))
=f(x)-f(y)$ whence (1).
\smallskip\\
$(1)\Rightarrow(2)$.
By induction on $t\in\N$ we define a sequence of functions
$\varphi_t\colon \{0,\ldots,t\}\to\N$ for $t\in\N$ such that $\varphi_{t+1}$ extends $\varphi_t$ and (*) and (**) below hold.
$$
\left\{\text{\begin{tabular}{cl}
(*)&\quad$\varphi_t$ is congruence preserving,\\
(**)&\quad $\pi_m(\varphi_t(u))= f(\pi_n(u))$ 
for all $u\in\{0,\ldots,t\}$.
\end{tabular}}\right.
$$
{\it Basis.}
We choose $\varphi_0(0)\in\N$ such that 
$\pi_m(\varphi_0(0))= f(\pi_n(0))$.
Properties  (*) and (**) clearly hold for $\varphi_0$.
\\
{\it Induction: from $\varphi_t$ to $\varphi_{t+1}$}.
Since the wanted $\varphi_{t+1}$ has to extend $\varphi_t$ to the domain
$\{0,\ldots,t,t+1\}$,
we only have to find a convenient value for $\varphi_{t+1}(t+1)$.
\\
By the induction hypothesis, (*) and (**) hold for $\varphi_t$;
in order for $\varphi_{t+1}$ to satisfy (*)  and (**), we have to find
 $\varphi_{t+1}(t+1)$  such that
$t+1-i$ divides $\varphi_{t+1}(t+1)-\varphi_{t}(i)$, for $i=0,\ldots,t$,
and $\pi_m(\varphi_{t+1}(t+1))= f(\pi_n(t+1))$.
Rewritten in terms of congruences,
these conditions amount to say that $\varphi_{t+1}(t+1)$ is a solution
of the following system of congruence equations:
\begin{equation}\label{system0}
\left.\begin{array}{r|rcll}
\text{$\star$(0)\ \qquad}
&\varphi_{t+1}(t+1)&\equiv& \varphi_t(0)&\quad\pmod {t+1}\\
&&{\footnotesize\vdots}\\
\text{$\star$(i)\ \qquad}
&\varphi_{t+1}(t+1)&\equiv& \varphi_t(i)&\quad\pmod {t+1-i}\\
&&{\footnotesize\vdots}\\
\text{$\star$(t-1)\ \qquad}
&\varphi_{t+1}(t+1)&\equiv& \varphi_t(t-1)&\quad\pmod 2\\
\text{$\star\star$\ \qquad}
&\varphi_{t+1}(t+1)&\equiv& \iota_m(f(\pi_n(t+1)))&\quad\pmod m
\end{array}\right\}
\end{equation}
Recall the Generalized Chinese Remainder Theorem
(cf. \S3.3, exercice 9 p. 114, in Rosen's textbook~\cite{rosen84}):
a system of congruence equations
$$
\bigwedge_{i=0,\ldots,t} x\equiv a_i\pmod {n_i}
$$
has a solution
if and only if $a_i\equiv a_j\mod\gcd(n_i,n_j)$ for all 
$0\leq i<j\leq t$.

Let us show that the conditions of application of 
the Generalized Chinese Remainder Theorem
are satisfied for system \eqref{system0}.
\begin{itemize}
\item
{Lines $\star$(i) and $\star$(j) of system \eqref{system0}
(with $0\leq i<j\leq t-1$).}\\
Every common divisor to $t+1-i$ and $t+1-j$ divides
their difference $j-i$ hence $\gcd(t+1-i,t+1-j)$ divides $j-i$.
Since $\varphi_t$ satisfies (*),  $j-i$ divides $\varphi_t(j)-\varphi_t(i)$
and a fortiori $\gcd(t+1-i,t+1-j)$ divides $\varphi_t(j)-\varphi_t(i)$.
\item
{Lines $\star$(i) and $\star\star$ of system \eqref{system0} (with $0\leq i\leq t-1$).}\\
Let $d=\gcd(t+1-i,m)$.
We have to show that $d$ divides $\iota_m(f(\pi_n(t+1)))-\varphi_t(i)$.
Since $f$ is congruence preserving, $\pi_{n,m}(\pi_n(t+1)-\pi_n(i))$ divides $f(\pi_n(t+1))-f(\pi_n(i))$.
As $m$ divides $n$, by Lemma \ref{l:utile}, $\pi_{n,m}(\pi_n(t+1)-\pi_n(i))=\pi_m(t+1)-\pi_m(i)=\pi_m(t+1-i)$ and
$f(\pi_n(t+1))-f(\pi_n(i))=k\pi_m(t+1-i)$ for some $k\in \Z/m\Z$.
Applying $\iota_m$,
there exists $\lambda\in\Z$ such that
\begin{equation*}%\label{eq:pf2}
\iota_m(f(\pi_n(t+1)))-\iota_m(f(\pi_n(i)))
=\iota_m(k)\iota_m(\pi_m(t+1-i))+\lambda m
\end{equation*}
  as   $\iota_m(\pi_m(u)) \equiv u\pmod m$  for every $u\in\Z$, there exists $\mu\in\Z$ such that
\begin{equation}\label{eq:pf3}
\iota_m(f(\pi_n(t+1)))-\iota_m(f(\pi_n(i)))
=\iota_m(k) (t+1-i) + \mu m + \lambda m\,.
\end{equation}
Since $\varphi_t$ satisfies (**), we have $\pi_m(\varphi_t(i)\!)\!=\!f(\pi_n(i)\!)$
\quad hence\\ $\varphi_t(i)\equiv\iota_m(f(\pi_n(i)))\pmod m$.
Thus equation \eqref{eq:pf3}  can be rewritten
\begin{equation}\label{eq:pf4}
\iota_m(f(\pi_n(t+1)))-\varphi_t(i)
=(t+1-i)\iota_m(k) +\nu m  \text{\quad for some $\nu$}\,.
\end{equation}
As $d$ divides $m$ and $t+1-i$, \eqref{eq:pf4} shows that  
$d$ divides $\iota_n(f(\pi_n(t+1)))-\varphi_t(i)$
as wanted.
\end{itemize}
Thus, we can apply the Generalized Chinese Theorem and get the wanted
value of $\varphi_{t+1}(t+1)$, concluding the induction step.
\smallskip\\
Finally, taking the union of the $\varphi_t$'s, $t\in\N$, we get a function
$F:\N\to\N$ which is congruence preserving and lifts $f$.
\end{proof}
\begin{example}[counterexample to Theorem \ref{thm:lift nm}] Lemma \ref{l:utile} and Theorem \ref{thm:lift nm} do not hold if $m$ does not divide $n$. Consider 
$f \colon \Z/6\Z \to \Z/8\Z$ defined by
$f(0)=0$, $ f(1)=3$, $f(2)=4$, $f(3)=1$, $f(4)=4$, $f(5)=7 $. %(See \cite{ch95}.) 
Note first that, in  $\Z/8\Z$, 1,3 and 5 are invertible, hence $f$ is congruence preserving iff for  $k\in\{2,\ 4\}$, for all $x\in \Z/6\Z$,  $k$ divides $f(x+k)-f(x)$ and this holds;
 nevertheless, $f$  has no congruence preserving lift $F\colon \Z \to\Z$. If such a lift $F$ existed, we should have
\begin{enumerate}
\item because  $F$ lifts $f$,  $\pi_8(F(0))\!=\!f(\pi_6(0))\!=\!\!0$ and $\pi_8(F(8))\!=\!f(\pi_6(8))\!=\!f(2)\!=\!4$; 
\item as $F$ is congruence preserving,
8 must divide $F(8) -F(0)$; we already noted that  8  divides $F(0)$, hence 8 divides $F(8)$ and $\pi_8(F(8))=0$, contradicting  $\pi_8(F(8))=4$. 
\end{enumerate}
Note that  $\pi_{6,8}$ is neither a homomorphism nor surjective and $0=\pi_{8}(8)\not=\pi_{6,8}\circ \pi_{6}(8)=2$.
\end{example}
As a first corollary of Theorem~\ref{thm:lift nm} we 
get a new proof of the representations of congruence preserving functions $\Z/n\Z\to\Z/m\Z$
as finite linear sums of polynomials with rational coefficients
(cf. \cite{cgg15}).
Let us recall the so-called binomial polynomials.

\begin{definition}\label{def:cpNZ}
For $k\in\N$, let
$
P_k(x)=\dbinom{x}{k}=\dfrac{1}{k!}\prod_{\ell=0}^{\ell=k-1} (x-\ell)$.
\end{definition}
Though $P_k$  has rational coefficients, it maps $\N$ into $\Z$.
Also, observe that $P_k(x)$ takes value $0$ for all $k>x$.
This implies that
for any sequence of integers $(a_k)_{k\in\N}$, the infinite sum
$
\sum_{k\in\N}a_k\,P_k(x)
$
reduces to a finite sum for any $x\in\N$ 
hence defines a function $\N\to\Z$. 

\begin{definition}\label{def:lcm}
We denote by $\lcm(k)$ the least common multiple of integers
$1,\ldots,k$ (with the convention $\lcm(0)=1$).
\end{definition}
\begin{definition}\label{def:PksurZ/nZ}
To each binomial polynomial $P_k$, $k\in\N$, we associate a function 
$P_k^{n,m}\colon\Z/n\Z\to\Z/m\Z$ which sends an element $x\in\Z/n\Z$ to
$(\pi_m\circ P_k\circ\iota_n)(x)\in\Z/m\Z$. 
\end{definition}
In other words,
consider the representative $t$ of $x$ lying in $\{0,\ldots,n-1\}$,
evaluate $P_k(t)$ in $\N$ and then take the class of the results in 
$\Z/m\Z$.
\begin{lemma}\label{akPk-is-CP}
If $lcm(k)$ divides $a_k$ in $\Z$, then the  function $\pi_m(a_k)P_k^{n,m}\colon \Z/n\Z \to \Z/m\Z$ (represented by $a_k P_k$)
is congruence preserving.
\end{lemma}
\begin{proof}  In \cite{cgg14} we proved that if $lcm(k)$ divides $a_k$ then $a_k P_k$ is a congruence preserving function on $\N$.
Let us now show  that $\pi_m(a_k)P_k^{n,m}\colon\Z/n\Z\to\Z/m\Z$ is also congruence preserving.
Let $x,y\in\Z/n\Z$: as $a_k P_k$ is congruence preserving, $\iota_n(x)-\iota_n(y)$ divides $a_k P_k(\iota_n(x))-a_k P_k(\iota_n(y))$. As $\pi_m$ is a morphism, $\pi_m(\iota_n(x))-\pi_m(\iota_n(y))$ divides $\pi_m(a_k) \pi_m(P_k(\iota_n(x)))-\allowbreak\pi_m(a_k)\pi_m( P_k(\iota_n(y)))=\pi_m(a_k)P_k^{n,m}(x)-\pi_m(a_k)P_k^{n,m}(x)$; as $\pi_m\circ \iota_n=\pi_{n,m}$ (Notation \ref{def:pi n}), we conclude that $\pi_m(a_k)P_k^{n,m}$ is congruence preserving. 
\end{proof}

\begin{corollary}[\cite{cgg15}]\label{cor:mainZ/nZ}
Let $1\leq m=p_1^{\alpha_1}\cdots p_\ell^{\alpha_\ell}$, $p_i$ prime.
Suppose $m$ divides $n$ and
let  $\nu(m)= \max_{i=1,\ldots,\ell}\;{p_i^{\alpha_i} }$.
A function $f\colon \Z/n\Z \to \Z/m\Z$ is  congruence preserving  
if and only if it is represented by  a finite $\Z$-linear sum such that $lcm(k)$ divides $a_k$ (in $\Z$) for all $k<\nu(m)$,  i.e.
$f=\sum_{k=0}^{\nu(m)-1} \pi_m(a_k) P_k^{n,m}$. 
\end{corollary}
\begin{proof} 
Assume $f\colon\Z/n\Z\to\Z/m\Z$ is congruence preserving.
Applying Theorem~\ref{thm:lift nm}, lift $f$ to $F\colon\N\to\N$
which is congruence preserving.
$$
\xymatrix{
\N    \ar[rrrr]^{\text{\normalsize$F=
                                \sum_{k=0}^{\nu(m)-1} a_k\, P_k$}}
        \ar[d]_{\text{\normalsize$\pi_n$}}
&&&& \Z \ar[d]^{\text{\normalsize$\pi_m$} \qquad
\begin{array}{rcl}
f\circ\pi_n&=&\pi_m\circ F
\end{array}}
\\
\Z/n\Z    \ar[rrrr]^{\text{\normalsize$f$}}
&&&&\Z/m\Z }
$$
We proved in \cite{cgg15} that every congruence preserving function $F\colon\N\to\N$ is of the form
$F= \sum_{k=0}^\infty a_k P_k$ where $\lcm(k)$ divides $a_k$
for all $k$. 
Since $F$ lifts $f$, for $u\in\Z$, we have
\begin{multline}\label{eq:f poly}
f(\pi_n(u))\ =\ \pi_m(F(u))
\ =\ \pi_m(\sum_{k=0}^\infty a_k\, P_k(u))
\\
=\ \sum_{k=0}^\infty \pi_m(a_k)\, \pi_m(P_k(u))
\ =\
\sum_{k=0}^{k=\nu(m)-1} \pi_m(a_k)\, \pi_m(P_k(u))
\end{multline}
The last equality is obtained by noting that for $k\geq \nu(m)$, 
$m$ divides $\lcm(k)$ hence 
as $a_k$ is a multiple of $\lcm(k)$, $\pi_m(a_k)=0$.
From \eqref{eq:f poly} we get 
$f(\pi_n(u))
=\sum_{k=0}^{k=\nu(m)-1} \pi_m(a_k)\, \pi_m(P_k(u))
=\pi_m(\sum_{k=0}^{k=\nu(m)-1} a_k\, P_k(u))$.
This proves that $f$ is lifted to the rational polynomial function 
$\sum_{k=0}^{k=\nu(m)-1} a_k\, P_k$.

\smallskip The converse follows from  Lemma \ref{akPk-is-CP} and the fact that  any finite sum of congruence preserving functions is congruence preserving.
\end{proof}

%\begin{remark} We can obtain another representation of congruence preserving functions
%$f\colon \Z/n\Z \to \Z/m\Z$ using Theorem~\ref{thm:lift}.
%Indeed we characterized congruence preserving functions $f\colon \Z \to \Z$ as infinite sums
%$\sum_{k\in\N}a_k\,Q_k(x)$, where $\lcm(k)$ divides $a_k$ for all $k$, and the $Q_k$ are rational polynomials defined as follows:
%$Q_0(x)=1$ and, for $k\in\N\setminus\{0\}$, $Q_k(x) =P_k(x+\left\lfloor\dfrac{k}{2}\right\rfloor)
%$, 
%i.e. $Q_{2s}(x) = \dfrac{1}{(2sk)!} 
%\,\prod_{\ell=-s}^{\ell=s-1} (x-\ell)$\,\,,\
%$Q_{2s+1}(x) = \dfrac{1}{(2s+1)!} 
%\,\prod_{\ell=-s}^{\ell=s} (x-\ell)$. 
%We can similarly represent  congruence preserving functions 
%$\Z/n\Z \to \Z/m\Z$ as projections of sums 
%$\sum_{k=0}^{k=\mu(m)-1}a_k\,Q_k$ with $\lcm(k)$ divides $a_k$.
%\hfill\qed\end{remark}
%%
\bigskip

As a second corollary of Theorem~\ref{thm:lift nm} we can lift  congruence preserving functions $\Z/n\Z\to\Z/n\Z$   to congruence preserving  functions $\Z/qn\Z\to\Z/qn\Z$.

We state a slightly more general result.
\begin{corollary}
Assume $m,n,q,r\geq1$,
$m$ divides both $n$ and $s$, and $n,s$ both divide $r$.
If $f\colon \Z/n\Z \to \Z/m\Z$ is congruence preserving
then it can be $(\pi_{r,n},\pi_{s,m})$-lifted to 
 $g\colon\Z/r\Z\to\Z/s\Z$ which is also congruence preserving.
 \end{corollary}
\begin{proof}
Using Theorem~\ref{thm:lift nm},
lift $f$ to a congruence preserving $F:\N\to\N$ and set 
$g=\pi_{s}\circ F\circ\iota_{r}$. 
We show that the following diagram commutes:
%\begin{figure}
\[\xymatrix{
\N    \ar[rrrrrr]^{\text{\normalsize$F$}}
	\ar@<-2pt>[rrd]_{\text{\normalsize$\pi_r$}}
        \ar[rrdd]_{\text{\normalsize$\pi_n$}}
&&&&&& \N    \ar[lld]_{\text{\normalsize$\pi_s$}}
               \ar[lldd]^{\text{\normalsize$\pi_m$}}
\\
&&\Z/r\Z    \ar[rr]_{\text{\normalsize$g$}}
                  \ar[llu]_{\text{\normalsize$\iota_r$}}
                \ar[d]^{\text{\normalsize$\pi_{r,n}$}}
&&\Z/s\Z  \ar[d]_{\text{\normalsize$\pi_{s,m}$}}
\\
&&\Z/n\Z \ar[rr]_{\text{\normalsize$f$}}
&&\Z/m\Z 
}\]
%\caption{Double lifting}\label{f:lifting}
%\end{figure}
$$
\begin{array}{rcll}
\pi_{s,m}\circ  g
&=& \pi_{s,m}\circ (\pi_{s}\circ F\circ\iota_{r})
\\
&=&(\pi_m\circ F)\circ\iota_{r}
&\text{by Lemma~\ref{l:utile} since $\pi_m=\pi_{s,m}\circ\pi_{s}$}
\\ 
&=& (f\circ\pi_n)\circ\iota_{r} &\text{ since \ $F$ lifts $f$}
\\
&=& f\circ\pi_{r,n} &\text{ since \ $\pi_n\circ\iota_{r}=\pi_{r,n}$}
\end{array}
$$

Thus, $\pi_{s,m}\circ g = f\circ \pi_{r,n}$, i.e. $g$ lifts $f$.

Finally, if $x,y\in\Z/r\Z$ then 
$\iota_r(x)-\iota_r(y)$ divides $F(\iota_r(x))-F(\iota_r(y))$
(by congruence preservation of $F$).
Since $\pi_s$ is a morphism and $\pi_s=\pi_{r,s}\circ\pi_r$, we deduce that 
$\pi_s(\iota_r(x))-\pi_s(\iota_r(y))=
(\pi_{r,s}\circ\pi_r\circ\iota_r)(x)-(\pi_{r,s}\circ\pi_r\circ\iota_r)(y)
=\pi_{r,s}(x-y)$ (recall $\pi_r\circ\iota_r$ is the identity on $\Z/r\Z$) divides
$\pi_s(F(\iota_r(x)))-\pi_s(F(\iota_r(y))=g(x)-g(y)$
(by definition of $g$).
Thus, $g$ is congruence preserving.
\end{proof}

\begin{remark}%{\color{red}{ Changement Attention}}
The previous diagram %of Figure \ref{f:lifting} 
is completely commutative:  $F$ lifts both $f$ and $g$,  and $g$ lifts $f$:
as $r$ divides $x-\iota_{r}\circ\pi_r(x)$ for all $x$, and $F$ is congruence preserving, $r$ divides $F(x)-F\circ\iota_{r}\circ\pi_r(x)$, and because $s$ divides $r$, $\pi_s\circ F(x)=\pi_{s}\circ F\circ\iota_{r}\circ\pi_r(x)$ hence
$\pi_s\circ F=g\circ\pi_r=\pi_{s}\circ F\circ\iota_{r}\circ\pi_r$. 
\end{remark}

%%
%%%%%%%%%%%%%%%%%%%%%%%%%%%%
%%%%%%%%%%%%%%%%%%%%%%%%%%%%
%%%%%%%%%%%%%%%%%%%%%%%%%%%%
\section{Congruence preservation on $p$-adic/profinite integers}
\label{s:profinite}
%%%%%%%%%%%%%%%%%%%%%%%%%%%%
%%%%%%%%%%%%%%%%%%%%%%%%%%%%
%%%%%%%%%%%%%%%%%%%%%%%%%%%%
%
All along this section, $p$ is a prime number;
we study congruence preserving functions on the
rings $\Z_p$ of $p$-adic integers
and  $\Zhat$ of profinite integers. 
$\Z_p$ is the projective limit 
${\underleftarrow{\lim}} \,\Z/p^n\Z$
relative to the projections $\pi_{p^n,p^m}$. 
Usually, $\Zhat$ is defined as the projective limit
${\underleftarrow{\lim}}\,\Z/n\Z$
of the finite rings $\Z/n\Z$ relative to the projections $\pi_{n,m}$,
for $m$ dividing $n$.
We here use the following equivalent definition which allows to
get completely similar proofs for $\Z_p$ and $\Zhat$.
$$
\Zhat\ =\ {\underleftarrow{\lim}}\  \Z/n!\Z
\ =\ \{\hat x=(x_n)_{n=1}^\infty\in\textstyle\prod_{n=1}^\infty\Z/n!\Z
\mid \forall m<n,\ x_m\equiv x_n\!\!\!\!\pmod {m!}\}
$$

Recall that $\Z_p$ (resp. $\Zhat$) contains the ring $\Z$
and is a compact topological ring
for the topology given by the ultrametric $d$ such that
$d(x,y)=2^{-n}$ where $n$ is largest such that $p^n$ (resp. $n!$) 
divides $x-y$, i.e. $x$ and $y$ have the same first $n$ digits
in their base $p$ (resp. base factorial) representation.
We refer to the Appendix for  some basic definitions,
representations and facts that we use about the compact topological rings
$\Z_p$ and $\Zhat$. 

\smallskip

\noindent We first prove that on $\Z_p$ and $\Zhat$  every congruence preserving function is continuous (Proposition \ref{p:unif cont}).
%%%%%%%%%%%%%%%%%%%%%%%%%%%
%\subsection{Congruence preserving functions are continuous}
%\label{ss:topology}
%%%%%%%%%%%%%%%%%%%%%%%%%%%
%
\begin{definition}
1. Let  $\mu:\N\to\N$ be increasing. 
A function $\Psi:\Z_p\to\Z_p$ admits $\mu$
as modulus of uniform continuity
if and only if $d(x,y)\leq2^{-\mu(n)}$ implies $d(\Psi(x),\Psi(y))\leq2^{-n}$.
\\
2. $\Phi$ is $1$-Lipschitz if it admits the identity as modulus of uniform continuity.
\end{definition}
Since the rings $\Z_p$ and $\Zhat$ are compact, every continuous 
function admits a modulus of uniform continuity.
\begin{proposition}\label{p:unif cont}
%For every $X\subseteq\Z_p$, 
Every congruence preserving function $\Psi:\Z_p\to\Z_p$
is $1$-Lipschitz. 
Idem with $\Zhat$ in place of $\Z_p$.
\end{proposition}
\begin{proof}
If $d(x,y)\leq 2^{-n}$ then $p^n$ divides $x-y$ hence
(by congruence preservation) $p^n$ also divides $\Psi(x)-\Psi(y)$
which yields $d(\Psi(x),\Psi(y))\leq 2^{-n}$.
\end{proof}
The converse of Proposition \ref{p:unif cont} is false: a continuous function is not necessarily congruence preserving as will be seen in Example \ref{CnonCP}. Note the following
%
%As a simple corollary, we get the existence of  non congruence preserving functions.
\begin{corollary}\label{cardinalite}
There are functions $\Z_p\to\Z_p$ (resp. $\Zhat\to\Zhat$)
which are not continuous hence not congruence preserving.
\end{corollary}
\begin{proof} As $\Z_p$ has cardinality $2^{\aleph_0}$
there are $2^{2^{\aleph_0}}$%>2^{\aleph_0}$ 
functions $\Z_p\to\Z_p$. 
Since $\N$ is dense in $\Z_p$, 
$\Z_p$ is a separable  space, hence
there are at most $2^{\aleph_0}$ continuous functions.
\end{proof}
%

%
%%%%%%%%%%%%%%%%%%%%%%%%%%%
%\subsection{Congruence preserving functions and inverse limits}
%\label{ss:cp are inverse}
%%%%%%%%%%%%%%%%%%%%%%%%%%%
%
%
In general an arbitrary continuous function on $\Z_p$
is not the inverse limit of a sequence of functions
$\Z/p^n\Z\to\Z/p^n\Z$'s. 
However, this is true for congruence preserving functions.
We first recall how any continuous function $\Psi\colon \Z_p\to \Z_p$
is the inverse limit of a sequence of an inverse system of
continuous functions $\psi_n\colon \Z/p^{\mu(n)}\Z\to\Z/p^n\Z$, $n\in\N$,
i.e. the diagrams of Figure~\ref{t:inverseLim} commute for any $m\leq n$.
For legibility, we use notations adapted to $\Z_p$:  we write
$\pi_n^p$ for $\pi_{p^n}\colon \Z_p\to\Z/p^n\Z$, $\pi_{n,m}^p$ (resp.  $\iota_{n,m}^p$)
\ for \ 
$\pi_{p^n,p^m}$ (resp. $\iota_{p^n,p^m}$), and $\iota_n^p$ for $\iota_{p^n}\colon \Z/p^n\Z\to \Z_p$.

\begin{proposition}\label{p:cont are inverse}
Consider $\Psi:\Z_p\to\Z_p$ and a strictly increasing $\mu:\N\to\N$.
Define \ $\psi_n:\Z/p^{\mu(n)}\Z\to\Z/p^n\Z$ \
as \ $\psi_n=\pi_n^p\circ\Psi\circ\iota_{\mu(n)}^p$
for all $n\in\N$.
\\
Then the following conditions are equivalent~:
\begin{enumerate}
\item
$\Psi$ is uniformly continuous %(hence uniformly so)
and admits $\mu$ as a modulus of uniform continuity.
%if  $x,y\in\Z_p$ and $d(x,y)\leq2^{-\mu(n)}$
%then $d(\Psi(x),\Psi(y))\leq2^{-n}$.
\item
For all $1\leq m\leq n$, the diagrams of Figure~\ref{t:inverseLim} commute
hence $\Psi$ is the inverse limit of the $\psi_n$'s, $n\in\N$.
\end{enumerate}
Idem with $\Zhat$ in place of $\Z_p$.
\end{proposition}
\begin{figure}$$
\begin{array}{c}
\xymatrix{
\Z/p^{\mu(n)}\Z 
\ar[r]^{\text{\normalsize$\iota^p_{\mu(n)}$}}
\ar[dr]^{\textit{\normalsize Id}}
&\Z_p  \ar[rr]^{\text{\normalsize$\Psi$}}
           \ar[d]^{\text{\normalsize$\pi^p_{\mu(n)}$}}
&&\Z_p  \ar[d]^{\text{\normalsize$\pi^p_n$}}
\\
&\Z/p^{\mu(n)}\Z  \ar[rr]^{\text{\normalsize$\psi_n$}}
                    \ar[d]^{\text{\normalsize$\pi^p_{\mu(n),\mu(m)}$}}
&&\Z/p^n\Z  \ar[d]^{\text{\normalsize$\pi^p_{n,m}$}}
\\
\Z/p^{\mu(m)}\Z  \ar[r]_{\textit{\normalsize Id}}
                 \ar[ur]^{\text{\normalsize$\iota^p_{\mu(m),\mu(n)}$}}
&\Z/p^{\mu(m)}\Z \ar[rr]_{\text{\normalsize$\psi_m$}}
&&\Z/p^m\Z }
\end{array}
$$
\caption{$\Psi$ as the inverse limit of the $\psi_n$'s, $n\in\N$.}\label{t:inverseLim}
\end{figure}
\begin{proof}
(1) and (2) are also equivalent to (3) below.

(3)
For all $1\leq m\leq n$, the lower half of the diagram
of Figure ~\ref{t:inverseLim} commutes.

\smallskip\noindent
$(1)\Rightarrow(2)$.
\textbullet\
We first show $\pi_n^p\circ\Psi=\psi_n\circ\pi_{\mu(n)}^p$.
Let $u\in\Z_p$.
Since $\pi_{\mu(n)}^p\circ\iota^p_{\mu(n)}$ is the identity 
on $\Z/p^{\mu(n)}\Z$,
we have $\pi_{\mu(n)}^p(u)
=\pi_{\mu(n)}^p(\iota^p_{\mu(n)}(\pi_{\mu(n)}^p(u)))$
hence $p^{\mu(n)}$ (considered as an element of $\Z_p$)
divides the difference 
$u-\iota^p_{\mu(n)}(\pi_{\mu(n)}^p(u))$,
i.e. the distance between these two elements is at most $2^{-\mu(n)}$.
As $\mu$ is a modulus of uniform continuity for $\Psi$,
the distance between their images under $\Psi$ is at most $2^{-n}$,
i.e. $p^n$ divides their difference, hence
$\pi_n^p(\Psi(u))
=\pi_n^p(\Psi(\iota^p_{\mu(n)}(\pi_{\mu(n)}^p(u))))$.
By definition, $\psi_n\!=\pi_n^p\circ\Psi\circ\iota^p_{\mu(n)}$.
Thus, $\pi_n^p(\Psi(u))\!
=\!\psi_n(\pi_{\mu(n)}^p(u))$, i.e. $\Psi$ lifts $\psi_n$.
\\
\textbullet\
We now show 
$\pi^p_{n,m}\circ\psi_n=\psi_m\circ\pi^p_{\mu(n),\mu(m)}$.
Since $\Psi$ lifts $\psi_m$, we have
\begin{eqnarray*}
\pi_m^p\circ\Psi &=&\psi_m\circ\pi_{\mu(m)}^p
\\
\text{hence\qquad}
\pi_m^p\circ\Psi\circ\iota^p_{\mu(n)}
&=& \psi_m\circ\pi_{\mu(m)}^p\circ\iota^p_{\mu(n)}
\\
\pi^p_{n,m}\circ\pi_n^p\circ\Psi\circ\iota^p_{\mu(n)}
&=&\psi_m\circ\pi^p_{\mu(n),\mu(m)}
\circ\pi^p_{\mu(n)}\circ\iota^p_{\mu(n)}
\\
\pi^p_{n,m}\circ\psi_n&=&\psi_m\circ\pi^p_{\mu(n),\mu(m)}
\quad\text{since $\pi^p_{\mu(n)}\circ\iota^p_{\mu(n)}$
is the identity.}
\end{eqnarray*}
This last equality means that $\psi_n$ lifts $\psi_m$.
\\
$(2)\Rightarrow(3)$. Trivial
\\
$(3)\Rightarrow(1)$.
The fact that $\Psi$ lifts $\psi_n$ shows that two elements of $\Z_p$
with the same first $\mu(n)$ digits (in the $p$-adic representation)
have images with the same first $n$ digits. This proves that $\mu$ is a
modulus of uniform continuity for $\Psi$.
\hfill\qed\end{proof}
For congruence preserving functions $\Phi:\Z_p\to\Z_p$,
the representation of Proposition~\ref{p:cont are inverse}
as an inverse limit gets smoother since then $\mu(n)=n$.
\begin{theorem}\label{thm:cp are inverse}
For a function $\Phi:\Z_p\to\Z_p$,
letting $\varphi_n:\Z/p^n\Z\to\Z/p^n\Z$ be  defined as
$\varphi_n=\pi_n^p\circ\Phi\circ\iota_n^p$,
the following conditions are equivalent.
\begin{itemize}
\item[(1)]
$\Phi$ is congruence preserving.
\item[(2)]
$\Phi$ is $1$-Lipschitz,
 all $\varphi_n$'s are congruence preserving
and $\Phi$ is the inverse limit of the $\varphi_n$'s, $n\in\N$.
\end{itemize}
A similar equivalence also holds for functions $\Phi:\Zhat\to\Zhat$.
\end{theorem}
\begin{figure}{
$$
\begin{array}{c}
\xymatrix{
\Z/p^n\Z 
\ar[r]^{\text{\normalsize$\iota^p_n$}}
\ar[dr]^{\textit{\normalsize Id}}
&\Z_p  \ar[rr]^{\text{\normalsize$\Phi$}}
           \ar[d]^{\text{\normalsize$\pi^p_n$}}
&&\Z_p  \ar[d]^{\text{\normalsize$\pi^p_n$}}
\\
&\Z/p^n\Z  \ar[rr]^{\text{\normalsize$\varphi_n$}}
                    \ar[d]^{\text{\normalsize$\pi^p_{n,m}$}}
&&\Z/p^n\Z  \ar[d]^{\text{\normalsize$\pi^p_{n,m}$}}
\\
\Z/p^m\Z  \ar[r]_{\textit{\normalsize Id}}
                 \ar[ur]^{\text{\normalsize$\iota^p_{m,n}$}}
&\Z/p^m\Z \ar[rr]_{\text{\normalsize$\varphi_m$}}
&&\Z/p^m\Z }
\end{array}
$$
\caption{$\Phi$ as the inverse limit of the $\varphi_n$'s, $n\in\N$.}\label{fig:PhinInverse}}
\end{figure}
\begin{proof}
(1) and (2) are also equivalent to (3) and (4) below.
\begin{itemize}
\item[(3)]
All $\varphi_n$'s are congruence preserving and,
for all $1\leq m\leq n$, the diagrams of Figure~\ref{fig:PhinInverse} commute.
\item[(4)]
All $\varphi_n$'s are congruence preserving and,
for all $1\leq m\leq n$, the lower half
(dealing with $\varphi_n$ and $\varphi_m$) of the diagrams of Figure~\ref{fig:PhinInverse}  commute.
\end{itemize}
\textbullet\
$(2)\Leftrightarrow(3)\Leftrightarrow(4)$.
Instantiate Proposition~\ref{p:cont are inverse}
with $\mu$ the identity on $\N$.
\\
\textbullet\
$(1)\Rightarrow(2)$.
Proposition~\ref{p:unif cont}  insures that $\Phi$ is $1$-Lipschitz.
We show that $\varphi_n$ is congruence preserving.
Since $\Phi$ is congruence preserving,
if $x,y\in\Z/p^n\Z$ then 
$\iota_n^p(x)-\iota_n^p(y)$ divides $\Phi(\iota_n^p(x))-\Phi(\iota_n^p(y))$.
Now, the canonical projection $\pi_n^p$ is a morphism hence
$\pi_n^p(\iota_n^p(x))-\pi_n^p(\iota_n^p(y))$ divides 
$\pi_n^p(\Phi(\iota_n^p(x)))-\pi_n^p(\Phi(\iota_n^p(y)))$.
Recall that   $\pi_n^p\circ\iota_n^p$ is the identity on $\Z/p^n\Z$.
Thus, $x-y$ divides
$\pi_n^p(\Phi(\iota_n^p(x)))-\pi_n^p(\Phi(\iota_n^p(y)))
=\varphi_n(x)-\varphi_n(y)$ as wanted.
\\
\textbullet\
$(4)\Rightarrow(1)$.
The fact that $\Phi$ lifts $\varphi_n$ shows that
two elements of $\Z_p$ with the same first $n$ digits
(in the $p$-adic representation)
have images with the same first $n$ digits. This proves that $\Phi$ is 
$1$-Lipschitz.
\\
It remains to prove that $\Phi$ is congruence preserving.
Let $x,y\in\Z_p$. Since $\varphi_n$ is congruence preserving
$\pi_n^p(x)-\pi_n^p(y)$ divides 
$\varphi_n(\pi_n^p(x))-\varphi_n(\pi_n^p(y))$.
Let
$$
U_n^{x,y}=\{u\in\Z/p^n\Z\mid 
\varphi_n(\pi_n^p(x))-\varphi_n(\pi_n^p(y))
=(\pi_n^p(x)-\pi_n^p(y))\, u\}\,.
$$
If $m\leq n$ and $u\in U_n^{x,y}$ then,
applying $\pi^p_{n,m}$ to the equality defining $U_n^{x,y}$,
and using the commutative diagrams of Figure~\ref{fig:PhinInverse} , we get
\begin{eqnarray*}
\varphi_n(\pi_n^p(x))-\varphi_n(\pi_n^p(y))
&=&(\pi_n^p(x)-\pi_n^p(y))\, u
\\
\pi^p_{n,m}(\varphi_n(\pi_n^p(x)))
-\pi^p_{n,m}(\varphi_n(\pi_n^p(y)))
&=&(\pi^p_{n,m}(\pi_n^p(x))- \pi^p_{n,m}(\pi_n^p(y)))\,
 \pi^p_{n,m}(u)
\\
\varphi_m(\pi^p_{n,m}(\pi_n^p(x)))
- \varphi_m(\pi^p_{n,m}(\pi_n^p(y)))
&=&(\pi^p_{n,m}(\pi_n^p(x))- \pi^p_{n,m}(\pi_n^p(y)))\,
 \pi^p_{n,m}(u)
\\
\varphi_m(\pi_m^p(x)) -\varphi_m(\pi_m^p(y))
&=&(\pi^p_m(x)- \pi^p_m(y))\, \pi^p_{n,m}(u)
\end{eqnarray*}
Thus, if $u\in U_n^{x,y}$ then $\pi^p_{n,m}(u)\in U_m^{x,y}$.

Consider the tree $\+T$ of finite sequences
$(u_0,\ldots,u_n)$ such that $u_i\in U^{x,y}_i$
and $u_i=\pi^p_{n,i}(u_n)$ for all $i=0,\ldots,n$.
Since each $U_n^{x,y}$ is nonempty, the tree $\+T$ is infinite.
Since it is at most $p$-branching, using K\"onig's  Lemma,
we can pick an infinite branch $(u_n)_{n\in\N}$ in $\+T$.
This branch defines an element $z\in\Z_p$.
The commutative diagrams of Figure~\ref{fig:PhinInverse} 
show that the sequences
$(\pi_n^p(x)-\pi_n^p(y))_{n\in\N}$
and $\varphi_n(\pi_n^p(x))-\varphi_n(\pi_n^p(y))$
represent $x-y$ and $\Phi(x)-\Phi(y)$ in $Z_p$.
Equalities
$\varphi_m(\pi^p_m(x))-\varphi_m(\pi^p_m(y))
=(\pi^p_m(x)-\pi^p_m(y))\ \pi_{n,m}(u)$ 
show that (going to the projective limits) $\Phi(x)-\Phi(y)=(x-y)\,z$.
This proves that $\Phi$ is congruence preserving.
\end{proof}
%
%
%%%%%%%%%%%%%%%%%%%%%%%%%%%
%\subsection{Extension of congruence preserving functions $\N\to\N$}
%\label{ss:extension}
%%%%%%%%%%%%%%%%%%%%%%%%%%%
%
%We have seen in the proof of Corollary \ref{cardinalite} supra 
%that there are not so many
%congruence preserving functions $\Zhat \to\Zhat$
%since their  number is at most $2^{\aleph_0}$
%(we shall see that equality holds).
%However, congruence preserving functions are not too sparse either.
%In fact,  any congruence preserving function $\Phi$ on $\Z_p$ or on $\Zhat$
%such that $\Phi(\N)\subseteq\N$ is determined by its trace on $\N$
%because $\N$ is dense in these rings.
Congruence preserving functions $\Zhat \to\Zhat$
are determined by their restrictions to $\N$ since $\N$ is dense in $\Zhat$.
Let us state a (partial) converse result. 
\begin{theorem}\label{extension}
Every congruence preserving function $F:\N\to\Z$ has a unique extension
to a congruence preserving function $\Phi:\Z_p\to\Z_p$
(resp. $\Zhat\to\Zhat$).
\end{theorem}
\begin{proof}
Observe that $\N$ is dense in $\Z_p$ (resp. $\Zhat$) and 
congruence preservation implies uniform continuity.
Thus, $F$ has a unique uniformly continuous extension $\Phi$
to $\Z_p$ (resp. $\Zhat$).
To show that this extension $\Phi$ is congruence preserving,
observe that $\Phi$ is the inverse limit of the
$\varphi_n=\rho_n\circ\Phi\circ\iota_n$'s.
Now, since $\iota_n$ has range $\N$, we see that
$\varphi_n=\rho_n\circ F\circ\iota_n$ hence is congruence preserving
as is $F$.
Finally, Theorem~\ref{thm:cp are inverse} insures that
$\Phi$ is also congruence preserving.
\end{proof}

\noindent Polynomials in $\Z_p[X]$ obviously define 
congruence preserving functions $\Z_p\to\Z_p$.
But non polynomial functions can  also  be congruence preserving.

\begin{consequence}
The extensions to $\Z_p$ and $\Zhat$ of 
the $\N\to\Z$ functions   \cite{cgg14,cgg14a} 
$$
x\mapsto \lfloor e^{1/a}\,a^x \, x!\rfloor
\text{\quad(for $a\in\Z\setminus\{0,1\}$)}
\quad,\quad
x\mapsto\texttt{if $x=0$ then $1$ else $\lfloor e\,x!\rfloor$}
$$
and the Bessel like function 
$
f(n)= 
\sqrt{\dfrac{e}{\pi}}
\times  \dfrac{\Gamma(1/2)}{2\times4^n\times n!}
\displaystyle\int_1^\infty e^{-t/2}(t^2-1)^n dt
$\ are congruence preserving.
\end{consequence}

%%%%%%%%%%%%%%%%%%%%%%%%%%%
%\subsection{Representation of congruence preserving functions}
%\label{ss:representation cp}
%%%%%%%%%%%%%%%%%%%%%%%%%%%
%

\noindent We now  characterize
 congruence preserving functions via their  representation as infinite linear sums of the $P_k$s;
 this representation is similar to Mahler's characterization  for  continuous functions (Theorem \ref{mahler}).
First recall the notion of valuation.% in $\Z_p$ (resp. $\Zhat$).
\begin{definition}
The $p$-valuation (resp. the factorial valuation) $\val(x)$
of $x\in\Z_p$, or $x\in\Z/p^n\Z$ (resp.  $x\in\Zhat$)
is the largest $s$ such that $p^s$ (resp. $s!$) divides $x$
or is $+\infty$ in case $x=0$.
It is also the length of the initial block of zeros
in the $p$-adic (resp. factorial) representation of $x$.
\end{definition}

Note that for any polynomial $P_k$ (or more generally any polynomial), the below diagram commutes for any $m\leq n$ (recall that $P_k^{p^n,p^n}=\ \pi_{p^n}\circ P_k\circ \iota_{p^n}$):
$$\quad\begin{CD}
 \Z/p^n\Z  @>\text{\normalsize$P_k^{p^n,p^n}$}>>   \Z/p^n\Z  \\
@V\text{\normalsize$\pi_{p^n,p^m}$}VV        @VV\text{\normalsize$\pi_{p^n,p^m}$}V\\
\Z/p^m\Z  @>\text{\normalsize$P_k^{p^m,p^m}$}>>  \Z/p^m\Z
\end{CD}\qquad\ \ \  {\text{ i.e.} }\quad
\pi_{p^n,p^m}\circ P_k^{p^n,p^n}=P_k^{p^m,p^m}\circ \pi_{p^n,p^m}$$

We now can define the interpretation $\widehat{P_k}(x)$ of ${P_k}(x)$ in $\Z_p$ (similar for $\Zhat$).
\begin{definition}\label{def:Pk Zp}
Define $\widehat{P_k}\colon\Z_p\to\Z_p$ as $\widehat{P_k}={\underleftarrow{\lim}}_{n\in\N} P_k^{p^n,p^n}$ . For $x\in\Z_p$, $x=({\underleftarrow{\lim}}_{n\in\N} x_n)$,  we have $\widehat{P_k}(x)\!=\!{\underleftarrow{\lim}}_{n\in\N}  \pi_{p^n}(P_k(\iota_{p^n}(x_n)\!)\!)$.
\end{definition}
Moreover, the below diagrams commute for all $n$

$$\begin{CD}
\Z_p@>\text{\normalsize$\widehat{P_k}$}>>   \Z_p \\
@V\text{\normalsize$\pi_n^p$} VV        @VV\text{\normalsize$\pi_n^p$} V\\
 \Z/p^n\Z      @>\text{\normalsize${P_k^{p^n,p^n}}$}>>   \Z/p^n\Z  \\
@V\text{\normalsize$\iota_{p^n}$}VV        @VV\text{\normalsize$\iota_{p^n}$}V\\
\N   @>\text{\normalsize$P_k$}>>  \N
\end{CD}
$$

\begin{theorem}[Mahler, 1956 \cite{mahler56}]\label{mahler}
1. A  series $\sum_{k\in\N} a_k\widehat{P_k} (x)$, $a_k\in\Z_p$, is convergent in $\Z_p$
if and only if  $\lim_{k\to\infty}a_k=0$, i.e.  the corresponding sequence of valuations $(\val(a_k))_{k\in\N}$  tends to $+\infty$.
\\
2. The above series represent all uniformly continuous functions
$\Z_p\to\Z_p$.
\\
Idem with $\Zhat$.
\end{theorem}
We can also characterize
of congruence preserving functions via their  representation as infinite linear sums of the $P_k$s.

 \begin{theorem}\label{thm:repCPZp}
 A function $\Phi:\Z_p\to\Z_p$ represented by a series $\Phi=\sum_{k\in\N}a_k\widehat{P_k}$ is congruence preserving if and only if
 $\lcm(k)$ divides $a_k$ for all $k$, i.e.  $a_k=
p^ib_k$ for $k\geq p^i$.% (in particular, the  $a_k$ tend to 0).
 \end{theorem}
\begin{proof} 
Suppose $\Phi$ is congruence preserving. By Theorem \ref{thm:cp are inverse}, $\Phi$ is uniformly continuous and by
Theorem \ref{mahler}, $\Phi=\sum_{k\in\N} a_k\widehat{P_k}$ with  $a_k\in\Z_p$.
Substituting in $\varphi_n=\pi^p_n\circ\Phi\circ\iota_n^p$, we get 
$\varphi_n=\pi^p_n\circ(\sum_{k\in\N} a_k\widehat{P_k})\circ\iota_n^p=
\sum_{k\in\N} \pi^p_n(a_k)\pi^p_n\circ \widehat{P_k} \circ\iota_n^p=
\sum_{k\in\N} \pi^p_n(a_k)P_k^{p^n,p^n}$\!\!.
Theorem~\ref{thm:cp are inverse} insures that $\Phi={\underleftarrow{\lim}}_{n\in\N} \varphi_n$ and the
$\varphi_n$ are congruence preserving on $\Z/p^n\Z$; thus by
Corollary~\ref{cor:mainZ/nZ}   :  $\varphi_n= \sum_{k=0}^{\nu(n)-1}b_k^nP_k^{p^n,p^n}$, with $\lcm(k)$
divides $b_k^n$ for all $k\leq \nu(n)-1$.
%hence for all $m\leq n$, $\pi_{p^n,p^m}\circ \phi_n=\phi_m\circ \pi_{p^n,p^m}$,
%and thus $\pi_{p^n,p^m}\circ \sum_{k=0}^{\nu(n)-1}b_k^nP_k^{p^n,p^n}=
%\sum_{k=0}^{\nu(n)-1}\pi_{p^n,p^m}(b_k^n) \pi_{p^n,p^m}\circ P_k^{p^n,p^n}=
%\sum_{k=0}^{\nu(m)-1}b_k^mP_k^{p^m,p^m}\circ \pi_{p^n,p^m}$,
%On the one hand, $\varphi_n=\pi^p_n\circ(\sum_{k\in\N} a_kP_k)\circ\iota_n^p=
%\sum_{k\in\N} \pi^p_n(a_k)\pi^p_n\circ\Phi\circ\iota_n^p=\sum_{k\in\N} \pi^p_n(a_k)P_k^{p^n,p^n}$.
We proved in \cite{cgg15} that the $P_k^{p^n,p^n}$ form a basis of the functions on $\Z/p^n\Z$, hence
$\pi^p_n(a_k)=b_k^n$ and $\lcm(k)$ divides $\pi^p_n(a_k)$. Noting  that   $\val(a_k)=\val(\pi^p_n(a_k)) $ and  applying Lemma \ref{l:lcmEtVal}, we deduce that $\lcm(k)$ divides $a_k$, i.e. $\nu_p(k)\leq \val(a_k)$, and $a_k=p^{\nu_p(k)}b_k$. In particular, this implies that $d(a_k,0) \leq 2^{-\nu_p(k)}$ and thus
$\lim_{k\to\infty}a_k=0$.

Conversely, if $\Phi=\sum_{k\in\N}a_k\widehat{P_k}$  and 
$\lcm(k)$ divides $a_k$ for all $k$, then $\lcm(k)$ divides $\pi^p_n(a_k)$ for all $n,k$;
hence the associated $\varphi_n$ are congruence preserving
which implies that so is $\Phi$.
\end{proof}
\begin{lemma} \label{l:lcmEtVal} Let $\nu_p(k)$ be the largest $i$ such that $p^i\leq k<p^{i+1}$. In $\Z/p^n\Z$, $\lcm(k)$ divides a number $x$ iff $\nu_p(k)\leq \val(x)$.
  \end{lemma}
\begin{proof} In $\Z/p^n\Z$ all numbers are invertible except multiples of $p$. Hence $\lcm(k)$ divides $x$
iff $p^{\nu_p(k)}$ divides $x$.\end{proof}

\begin{example} \label{CnonCP}
Let $\Phi=\sum_{k\in\N}a_k\,P_k$ with $a_k=p^{\nu_p(k)-1}$,  with $\nu_p(k)$ as in Lemma  \ref{l:lcmEtVal}.  $\Phi$ is uniformly continuous by Theorem \ref{mahler}. By Lemma  \ref{l:lcmEtVal} $\lcm(k)$ does not divide $a_k$ hence  by Theorem \ref{thm:repCPZp} $\Phi$ is {\em not }congruence preserving.
\end{example}
%
%It is well known that $ \Zhat=\prod_{p\  {\text prime}} {\Z_p}$ (\cite{lang}, exercise 44 page 329), hence  Theorem \ref{thm:repCPZp} also enables us to characterize  congruence preserving functions on $\Zhat$.

%%%%%%%%%%%%%%%%%%%%%%%%%%%%
%%%%%%%%%%%%%%%%%%%%%%%%%%%%
%%%%%%%%%%%%%%%%%%%%%%%%%%%%
\section{Congruence preserving functions and lattices} 
\label{s:lattices}
%%%%%%%%%%%%%%%%%%%%%%%%%%%%
%%%%%%%%%%%%%%%%%%%%%%%%%%%%
%%%%%%%%%%%%%%%%%%%%%%%%%%%%
%
%In this section, we extend Theorem \ref{thm:MAIN} of our paper \cite{cgg14ipl} to functions $\Z\to\Z$.

A {\em lattice } of subsets of a set $X$ is a family of subsets of $X$
such that $L\cap M$ and $L\cup M$ are in $\+L$
whenever $L,M\in \+L$.
Let $f:X\to X$.
A lattice $\+L$ of subsets of $X$ is {\it closed} under $f^{-1}$ if
$f^{-1}(L)\in\+L$ whenever $L\in\+L$.
{\em Closure under decrement } means closure under $\suc^{-1}$, 
where $\suc$ is the successor function.
%Let  $\+P(X)$ denote .
For $L\subseteq\Z$ and $t\in\Z$, let $L-t=\{x-t\mid x\in L\}$.
\begin{proposition}\label{p:LXL}
Let $X$ be $\N$ or $\Z$ or 
$\N_\alpha=\{x\in\Z\mid x\geq\alpha\}$ with $\alpha\in\Z$.
For $L$ a %finite
 subset of $ X$ let $\+L_X(L)$ be the family of sets of the form\
$\bigcup_{j\in J}\bigcap_{i\in I_j}X\cap(L-i)$
where $J$ and the $I_j$'s are finite non empty subsets of $\N$.
Then $\+L_X(L)$ is the smallest sublattice of  $\+P(X)$, the class of subsets of $X$, %$\FINX$
containing $L$ and closed under decrement.
\end{proposition}
%
%The following characterization is proved in \cite{cgg14ipl}:
%
%\begin{theorem}\label{thm:MAIN}
Let $f:\N\longrightarrow\N$ be a non decreasing function, 
the following conditions are equivalent  \cite{cgg14ipl}:

$(1)_\N$\quad 
For every finite subset $L$  of $\N$, the lattice $\+L_{\N}(L)$ is closed under $f^{-1}$.

$(2)_\N$\quad 
The function $f$ is congruence preserving and
$f(a)\geq a$ for all $a\in\N$.

$(3)_\N$\quad 
For every  regular subset $L$ of $\N$ the lattice $\+L_{\N}(L)$ is  closed under $f^{-1}$.

\smallskip\noindent
%\end{theorem}
%
We now extend this result to functions $\Z\to\Z$.
First recall the notions of {\it recognizable} and {\it rational} subsets:
a subset $L$ of a monoid $X$ is rational if  it can be generated from  finite sets by unions, products and stars;
$L$ is recognizable if there exists a morphism $\varphi \colon X \longrightarrow M$, with $M$ a finite monoid, and $F$ a  subset of $M$ such that $L=\varphi^{-1}(F)$. For $\N$,  recognizable and rational subsets
coincide and are 
 are called regular subsets of $\N$. For $\Z$, recognizable subsets are finite unions of arithmetic
sequences, while rational subsets are unions of the form $F\cup P\cup -N$, with $F$ finite, 
and $P,N$ two regular subsets of $\N$; i.e. a recognizable subset of $\Z$ is also rational, but the converse is false.
%
%\begin{definition}\label{def:reg}
%Conventions $\min\emptyset=+\infty$
%and $\max\emptyset=-\infty$.
%\\
It is known that

{\it 
1. A subset $L\subseteq\N$ is regular
if it is the union of a finite set with finitely many arithmetic progressions,
i.e. $L=F\cup(R+d\N)$ with
$d\geq1$, $F,R\subseteq\{x\mid0\leq x<d\}$
(possibly empty).

2. A subset $L\subseteq\Z$ is rational if it is of the form
$L=L^+\cup (-L^-)$ where $L^+,L^-$ are regular subsets of $\N$,
i.e. $L=-(d+S+d\N)\cup F\cup(d+R+d\N)$ with
$d\geq1$,
$R,S\subseteq\{x\mid0\leq x<d\}$,
$F\subseteq\{x\mid-d<x<d\}$
(possibly empty).See \cite{benois}.

3. A subset $L\subseteq\Z$ is recognizable if it is of the form
$L=(F+d\Z)$ with
$d\geq1$, $F\subseteq\{x\mid0\leq x<d\}$
}
%\end{definition}
%
%We can now extend Theorem \ref{thm:MAIN} to functions $\Z\to\Z$.
%
\begin{theorem}\label{thm:mainZ1}
Let $f:\Z\longrightarrow\Z$ be a non decreasing function.
The following conditions are equivalent:
%\begin{enumerate}%\marginpar{\color{red} Proof in\\ Appendix}

\noindent $(1)_\Z$\quad 
For every finite subset $L$  of $\Z$,
the lattice $\+L_\Z(L)$ is closed under $f^{-1}$.

\noindent $(2)_\Z$\quad 
The function $f$ is congruence preserving and
$f(a)\geq a$ for all $a\in\Z$.

\noindent $(3)_\Z$\quad For every  recognizable subset $L$ of $\Z$ the lattice $\+L_{\Z}(L)$ is  closed under $f^{-1}$.
\end{theorem}
\begin{proof}%[of Theorem \ref{thm:mainZ1}]
For this proof, 
we need the $\Z$-version of  Lemma 3.1 in \cite{cgg14ipl}.
\begin{lemma}\label{l:f-1L}
Let $f :\Z\to\Z$ be a nondecreasing congruence preserving function.
Then, for any set $L \subseteq \Z$,
we have
$f^{-1}(L) \ =\ {\textstyle\bigcup_{a\in f^{-1}(L)}\bigcap_{t\in L-a} (L-t)}$\,.
\end{lemma}
\begin{proof}
Let $a \in f^{-1}(L)$.
As $t \in L-a\Leftrightarrow a \in L-t$,
we have $a\in\bigcap_{t\in L-a} L-t$,
proving inclusion $\subseteq$.

For the other inclusion, let $b \in\bigcap_{t\in L-a} L-t$
with $a \in f^{-1}(L)$.
To prove that $f(b)\in L$, we argue by way of contradiction.
Suppose $f(b)\notin L$.
Since $f(a)\in L$ we have $a\neq b$.
The condition on $f$ insures the existence of $k \in\Z$ such that
$f(b)-f(a) = k(b-a)$. In fact, $k\in\N$ since $f$ is nondecreasing.

Suppose first that $a < b$.
Since $k\in\N$ and $f(a)+k(b-a)=f(b)\notin L$
there exists a least $r \in\N$ such that $f(a) + r(b-a) \notin  L$.
Moreover, $r\geq 1$ since $f(a) \in L$.
Let $t = f(a) -a+ (r-1)(b-a)$.
By minimality of $r$, we get $t+a=f(a) + (r-1)(a-b) \in L$.
Now $t + a \in L$ implies $t  \in L-a$; as $b \in\bigcap_{t\in L-a} L-t$ this implies $b\in L-t$ hence  $t+b \in L$.
But $t+b=f(a) + r(b-a)\notin L$, this contradicts the definition of $r$.

Suppose next that $a > b$. 
Since $k\in\N$ and $f(b)+k(a-b)=f(a)\in L$
there exists a least $r \in\N$ such that $f(b) + r(a-b) \in  L$.
Moreover, $r\geq 1$ since $f(b) \notin L$.
Let $t = f(b)-b + (r-1)(a-b)$.
By minimality of $r$, we get $t+b=f(b) + (r-1)(a-b) \notin L$.
Now $t + a \in L$ implies $t+b \in L$, contradiction.\end{proof}

{$\bullet\ \ (1)_\Z \Rightarrow (2)_\Z$.}
Assume $(1)_\Z$ holds.
We first prove inequality $f(x)\geq x$ for all $x\in\Z$.
Observe that (by Proposition~\ref{p:LXL})
$\+L_\Z(\{z\})=\{X\in\FINZ\mid X=\emptyset\text{ or }\max X\leq z\}$.
In particular, letting $z=f(x)$ and applying $(1)_\Z$ with $\+L(\{f(x)\})$,
we get $f^{-1}(\{f(x\})\in\+L_\Z(\{f(x)\})$
hence $x\leq\max(f^{-1}(\{f(x\}))\leq f(x)$.

To show that $f$ is congruence preserving, we reduce to the $\N$ case.

For $\alpha\in\Z$, let $\suc_\alpha:\N_\alpha\to\N_\alpha$
be the successor function on $\N_\alpha=\{z\in\Z\mid z\geq\alpha\}$.
The structures $\langle\N,\suc\rangle$ and
$\langle\N_\alpha,\suc_\alpha\rangle$ are isomorphic.
Since $f(x)\geq x$ for all $x\in\Z$, the restriction $f\segment\N_\alpha$
maps $\N_\alpha$ into $\N_\alpha$.
In particular, using our result in $\N$,
conditions $(1)_{\N_\alpha}$ and  $(2)_{\N_\alpha}$
(relative to $f\segment\N_\alpha$) are equivalent.

We show that condition $(2)_{\N_\alpha}$ holds.
Let $L\subseteq\N_\alpha$ be finite.
Condition $(1)_\Z$ insures that $\+L_\Z(L)$ is closed under $f^{-1}$.
In particular, $f^{-1}(L)\in\+L_\Z(L)$. Using Proposition~\ref{p:LXL},
we get
$f^{-1}(L)\ =\ \textstyle\bigcup_{j\in J}\bigcap_{i\in I_j} (L-i)$
for finite $J$, $I_j$'s included in $\N$
hence
$(f\segment\N_\alpha)^{-1}(L)\ =\ f^{-1}(L)\cap\N_\alpha
\ =\ \bigcup_{j\in J}\bigcap_{i\in I_j}(\N_\alpha\cap (L-i))
\in\+L_{\N_\alpha}(L)$.
This proves condition $(1)_{\N_\alpha}$.
Since $(1)_{\N_\alpha}\Rightarrow(2)_{\N_\alpha}$
we see that $f\segment\N_\alpha$ is congruence preserving
Now, $\alpha$ is arbitrary in $\Z$ and the fact that
$f\segment\N_\alpha$  is congruence preserving for all $\alpha\in\Z$ implies that 
 $f$ is congruence preserving.
Thus,  condition $(2)_\Z$ holds.
\smallskip\\
{$\bullet\ \ (2)_\Z \Rightarrow(3)_\Z $. }
Assume $(2)_\Z$. It is enough to prove that $f^{-1}(L)\in\+L_{\Z}(L)$ whenever $L$  is recognizable.
Let  $L=(F+d\Z)$ with
$d\geq1$, $F=\{f_1,\cdots,f_n\}\subseteq\{x\mid0\leq x<d\}$. 
Then $f$ is not constant since $f(x)\geq x$ for all $x\in\Z$.
Also, $f^{-1}(\alpha)$ is finite for all $\alpha$~:
let $b$ be such that $f(b)=\beta\neq\alpha$,
by  congruence preservation
the nonzero integer $\alpha-\beta$ is divided by $a-b$ for all $a\in f^{-1}(\alpha)$
hence $f^{-1}(\alpha)$ is finite. $f^{-1}(F)$ is thus finite too. Moreover, 
$L-t=F-t+d\Z= L-t-d+d\Z=L-t-d=L-t+d+d\Z=L-t+d$, hence
  there are only  finitely many $L-t$ 's.
By Lemma~\ref{l:f-1L} we have
$f^{-1}(L)  \ =\ 
\bigcup_{a\in f^{-1}(F)} \bigcap_{t\in L-a}( L-t)$; as there are only a finite number of $L-t$ 's, 
all union and intersections reduce to finite unions and intersections and $f^{-1}(L)\in\+L_{\Z}(L)$.

\noindent{$\bullet\ \ (3)_\Z \Rightarrow(2)_\Z $. } Similar to {$(1)_\Z \Rightarrow(2)_\Z $. }

\noindent{$\bullet\ \ (2)_\Z \Rightarrow(1)_\Z $. } Similar to {$(2)_\Z \Rightarrow(3)_\Z $. }\qed\end{proof}

\begin{example} 
Theorem \ref{thm:mainZ1}  does not hold if we substitute  rational for recognizable in $(3)_\Z$.
 Consider
$L= (6+10 \N)$ and $f(x) = x^2$; $L$ is rational and $f$ is congruence preserving.
However $f^{-1}(L) = (\{4,6\}+10\N)\cup -(\{4,6\}+10\N)$ does not belong to $\+L_{\Z}(L)$:
$f^{-1}(L)$ contains infinitely many negative numbers, while each $L-t$ for $t\in f^{-1}(L)$ contains only finitely many negative numbers; hence any finite union of finite intersections of $L-t$ 's can contain only a finite number of negative numbers and cannot be equal to 
$f^{-1}(L)$.
\qed\end{example}

 Theorem \ref{thm:mainZ1} does not hold for $ \Z_p$: $f^{-1}(L)$ no longer belongs to $\+L_{\Z_p}(L)$ , the lattice of subsets of $\Z_p$ containing $L$ and closed under decrement.
 Consider the congruence preserving function $f(x)=\big(\sum_{i\geq 2} p^i\big) x$, and  let $L=\{\sum_{i\geq 2} p^i\}=\{f(1)\}$. Then $f^{-1}(L)=\{1\}\not\in \+L_{\Z_p}(L)$ because all elements of the $(L-i)$\,s end with an infinity of 1\,s.
 
 Thus integer decrements are not sufficient; but even if we substitute  translations for decrements,
 Theorem \ref{thm:mainZ1} can't be generalized.

A recognizable subset of  $ \Z_p$ is of the form  $F + p^n \Z_p$ with $F$ finite,  $F\subseteq \Z/p^n \Z$.  For $L$ a %finite
 subset of $ \Z_p$ let $\+L^c_{\Z_p}(L)$ (resp. $\+L_{\Z_p}(L)$)  be the family of sets of the form\
$\bigcup_{j\in J}\bigcap_{i\in I_j}(L+a_i)$, $a_i\in\Z_p$,
where $J$ and the $I_j$'s are (resp. finite)  non empty subsets of $\N$.
Then $\+L^c_{\Z_p}(L)$ (resp. $\+L_{\Z_p}(L)$)  is the smallest complete sublattice (resp. sublattice) of  $\+P(\Z_p)$, the class of subsets of $\Z_p$, %$\FINX$
containing $L$ and closed under translation.
It is easy to see that, 
%\begin{lemma}\label{l:f-aL}
%Let 
for $f :\Z_p\to\Z_p$.
and $L \subseteq \Z/p^k\Z$,
we have
$f^{-1}(L) \ =\  {\textstyle\bigcup_{a\in f^{-1}(L)}\bigcap_{t\in L} (L+(a-t))}$\,.
%\end{lemma}
%
%\begin{proof} For $L \subseteq \Z/p^k\Z$,  $\cap_{t\in L} (L-t)=\{0\}$ ; hence
%$$\cup_{a\in f^{-1}(L)}\cap_{t\in L} (L+(a-t))= \{a\}$$ and the result.
%\qed\end{proof}
%
%\begin{theorem}\label{thm:mainZp}
%Let $f:\Z_p\longrightarrow\Z_p$ be a congruence preserving function. 
Hence for any $f :\Z_p\to\Z_p$

$(1)_{\Z_p}$ \  If  $f^{-1}(a)$ is finite for every $a$, then
for every finite subset $L$  of $\Z_p$,
the lattice $\+L_{\Z_p}(L)$ is closed under $f^{-1}$.

 $(3)_{\Z_p}$\ 
For every  recognizable subset $L$ of ${\Z_p}$ the lattice $\+L^c_{\Z_p}(L)$ is  closed under $f^{-1}$.

%\end{theorem}
%
%\begin{proof} Follows from Lemma \ref{l:f-aL}.\qed
%\end{proof}

%\begin{example}  
However conditions $(1)_{\Z_p}$ and $(3)_{\Z_p}$ %of Theorem \ref{thm:mainZp}  
do not imply that $f$ is congruence preserving: 
 let $f$  be inductively defined on $\Z$ by: for $0\leq x<p$, $f(x)=x$, and for $x\geq p$, $f(x) = f(x-p)+1$. For $np\leq k < (n+1)p$, $f(k)=n+(k-np)$; hence $f$ is uniformly continuous and has a unique uniformly continuous 
 extension $\hat f$ to $\Z_p$;  then 
 $\hat f$ satisfies $(1)_{\Z_p}$ and $(3)_{\Z_p}$but is not congruence preserving as $p$ does not divide 1.
%\qed\end{example}

%%%%%%%%%%%%%%%%%%%%%%%%%%%%
%%%%%%%%%%%%%%%%%%%%%%%%%%%%
%%%%%%%%%%%%%%%%%%%%%%%%%%%%
\section{Conclusion}
We here studied functions having  congruence preserving properties; these functions appeared in two ways at least: 
{\it  (i)}\/
as the functions such that lattices of regular subsets of $\N$ are closed under $f¬^{-1}$ (see \cite{cgg14ipl}), and
{\it  (ii)}\/
as the functions uniformly continuous in a variety of finite groups (see \cite{PinSilva2011}).

 The contribution of the present paper  is   to {\em characterize  congruence preserving functions} 
on various sets derived from $\Z$   such as  $\Z/n\Z,$  (resp. $\Z_p,\widehat{\Z}$) via polynomials (resp. series) with {\em rational coefficients} which share the following  common property:  $\lcm(k)$ divides the $k$-th coefficient.  Examples of  {\em non polynomial} (Bessel like) congruence preserving functions can be found in \cite{cgg14a}.
\section*{Acknowledgments} We thanks the anonymous referee for careful reading and valuable comments.
%%%%%%%%%%%%%%%
%%%%%%%%%%%%%%%%%
%%%%%%%%%%%%%%%%%%%%

%\fi
%%%%%%%%%%%%%%%%%%%%%%%%%%%%
%%%%%%%%%%%%%%%%%%%%%%%%%%%%
%%%%%%%%%%%%%%%%%%%%%%%%%%%%
%%%%%%%%%%%%%%%%%%%%%%%%%%%%
%%%%%%%%%%%%%%%%%%%%%%%%%%%%
%{\small

%
%%%%%%%%%%%%%%%%%%%%%%%%%%%%
%%%%%%%%%%%%%%%%%%%%%%%%%%%%
%%%%%%%%%%%%%%%%%%%%%%%%%%%%
\section{Appendix}
\label{s:appendix}
%%%%%%%%%%%%%%%%%%%%%%%%%%%%
%%%%%%%%%%%%%%%%%%%%%%%%%%%%
%%%%%%%%%%%%%%%%%%%%%%%%%%%%
%
%%%%%%%%%%%%%%%%%%%%%%%%%%%%
%\subsection{Basics on $p$-adic and profinite integers}
%\label{ss:profinite}
%%%%%%%%%%%%%%%%%%%%%%%%%%%%
%
Recall some classical equivalent approaches to
the topological rings of $p$-adic integers and profinite integers,
cf. Lenstra \cite{lenstra05,lenstra05a}, Lang \cite{lang} and 
Robert \cite{robert00}.
\begin{proposition}\label{p:Zp}
Let $p$ be prime.
The three following approaches lead to isomorphic structures,
called the topological ring $\Z_p$ of $p$-adic integers.
\begin{itemize}
\item
The ring $\Z_p$ is the inverse limit of the following inverse system:
\begin{itemize}
\item
the family of rings $\Z/p^n\Z$ for $n\in\N$,
endowed with the discrete topology,
\item
the family of surjective morphisms
$\pi_{p^n,p^m}:\Z/p^n\Z \to \Z/p^m\Z$ for $0\leq n\geq m$.
\end{itemize}
%----------------------
\item
The ring $\Z_p$ is the set of infinite sequences $\{0,\ldots,p-1\}^\N$
endowed with the Cantor topology and addition and multiplication
which extend the usual way to perform addition and multiplication
on base $p$ representations of natural integers.
%----------------------
\item
The ring $\Z_p$ is the Cauchy completion of the metric topological ring
$(\N,+,\times)$
relative to the following ultrametric: 
$d(x,x)=0$ and for $x\neq y$, $d(x,y)=2^{-n}$ where $n$ is the $p$-valuation of $|x-y|$,
i.e. the maximum $k$ such that $p^k$ divides $x-y$.
\end{itemize}
\end{proposition}
Recall the factorial representation of integers.
\begin{lemma}\label{def:factorial}
Every positive integer $n$ has a unique representation as
$$
n=c_k k! + c_{k-1} (k-1)! +...+ c_2 2! + c_1 1!
$$
where $c_k\neq 0$ and $0\leq c_i \leq i$ for all $i=1,...,k$.
\end{lemma}
%
%\begin{proof}
%{\em Existence.}
%By induction on $n$.
%Let $k$ be greatest such that $n\geq k!$.
%Then $n=c_k k! +m$ where $c_k<k+1$ and $m<k!$.
%\\
%{\em Uniqueness.}
%Observe, by induction on $k$, that $\sum_{i<k}i\ i! < k!$.
%\end{proof}
%
%
\begin{proposition}\label{p:Zph}
The four following approaches lead to isomorphic structures,
called the topological ring $\Zhat$ of profinite integers.
\begin{itemize}
\item
The ring $\Zhat$ is the inverse limit of the following inverse system:
\begin{itemize}
\item
the family of rings $\Z/k\Z$ for $k\geq1$,
endowed with the discrete topology,
\item
the family of surjective morphisms
$\pi_{n,m}:\Z/n\Z \to \Z/m\Z$ for $m\divi n$.
\end{itemize}
%----------------------
\item
The ring $\Zhat$ is the inverse limit of the following inverse system:
\begin{itemize}
\item
the family of rings $\Z/k!\Z$ for $k\geq1$,
endowed with the discrete topology,
\item
the family of surjective morphisms
$\pi_{(n+1)!,n!}:\Z/n!\Z \to \Z/m!\Z$ for $n\geq m$.
\end{itemize}
%----------------------
\item
The ring $\Zhat$ is the set of infinite sequences
$\prod_{n\geq1}\{0,\ldots,n\}$
endowed with the product topology and addition and multiplication
which extend the obvious way to perform addition and multiplication
on factorial representations of natural integers.
%----------------------
\item
The ring $\Zhat$ is the Cauchy completion of the metric topological ring
$(\N,+,\times)$
relative to the following ultrametric: for $x\neq y\in\N$,
$d(x,x)=0$ and $d(x,y)=2^{-n}$ where $n$ is the maximum $k$
such that $k!$ divides $x-y$.
%----------------------
\item
The ring $\Zhat$ is the product ring $\prod_{p\text{ prime}}\Z_p$
endowed with the product topology.
\end{itemize}
\end{proposition}

\begin{proposition}\label{p:Zp compact}
The topological rings $\Z_p$ and $\Zhat$ are compact and zero dimensional (i.e. they have a basis of closed open sets).
\end{proposition}

%
%%%%%%%%%%%%%%%%%%%%%%%%%%%%
%\subsection{$\N$ and $\Z$ in $\Z_p$ and $\Zhat$}
%\label{ss:NZ in Zp}
%%%%%%%%%%%%%%%%%%%%%%%%%%%%
%
\begin{proposition}\label{p:embed N}
Let $\lambda:\N\to\Z_p$ (resp. $\lambda:\N\to\Zhat$)
be the function which maps $n\in\N$ to the element of $\Z_p$
(resp. $\Zhat$)
with base $p$ (resp. factorial) representation obtained by suffixing
an infinite tail of zeros to the base $p$ (resp. factorial) representation
of $n$.
\\
The function $\lambda$ is an embedding of the semiring $\N$ onto 
a topologically dense semiring in the ring $\Z_p$
(resp. $\Zhat$).
\end{proposition}

\begin{remark}\label{rk:opposite}
In the base $p$ representation,
the opposite of an element $f\in\Z_p$ is the element $-f$ such that,
for all $m\in\N$,
$$
(-f)(i) = \left\{\begin{array}{ll}
0 & \textit{\quad if \ }\forall s\leq i\ f(s)=0\,,\\
p-f(i) & \textit{\quad if \ $i$ is least such that $f(i)\neq0$}\,,\\
p-1-f(i) & \textit{\quad if \ }\exists s<i\ f(s)\neq0\,.
\end{array}\right.
$$
In particular, 
\\-- 
Integers in $\N$ correspond in $\Z_p$ to infinite base $p$
representations with a tail of $0$'s.
\\--
Integers in $\Z\setminus\N$ correspond in $\Z_p$ to infinite base $p$
representations with a tail of digits $p-1$.
%\end{itemize}

\noindent Similar results hold for the infinite factorial representation
of profinite integers.
\end{remark}

%}
\end{document}